\documentclass[11pt,reqno]{amsart}
\usepackage{enumerate}
\usepackage{enumitem}
\usepackage{mathrsfs}
\usepackage{url}
\usepackage{cite}
\usepackage{amsmath}
\usepackage{comment}
\usepackage{graphicx}
\usepackage{upgreek}
\usepackage{tikz}
\usepackage{amsfonts,amssymb}
\usetikzlibrary{arrows,shapes,trees,backgrounds}
\usepackage{fancyhdr}
\usepackage{amsfonts, amsmath, amssymb, amscd, amsthm, bm, cancel}
\usepackage[numbers,sort&compress]{natbib}

\usepackage[
linktocpage=true,colorlinks,citecolor=magenta,linkcolor=blue,urlcolor=magenta]{hyperref}
\usepackage[margin=1.1in]{geometry}

\newtheorem{proposition}{Proposition}[section]
\newtheorem{theorem}[proposition]{Theorem}
\newtheorem{lemma}[proposition]{Lemma}
\newtheorem{corollary}[proposition]{Corollary}

\theoremstyle{definition}
\newtheorem*{ack}{Acknowledgements}
\theoremstyle{remark}

\numberwithin{equation}{section}
\allowdisplaybreaks

\begin{document}

\title[Gauss curvature type flow and Alexandrov-Fenchel inequalities]{Gauss curvature type flow and Alexandrov-Fenchel inequalities in the hyperbolic space}

\author[T. Luo]{Tianci Luo}
\address{School of Mathematical Sciences, University of Science and Technology of China, Hefei 230026, P.R. China}
\email{\href{mailto:Luo_tianci@mail.ustc.edu.cn}{Luo\_tianci@mail.ustc.edu.cn}}
\author[R. Zhou]{Rong Zhou}
\email{\href{mailto:zhourong@mail.ustc.edu.cn}{zhourong@mail.ustc.edu.cn}}
\date{\today}

\begin{abstract}
We consider the Gauss curvature type flow for uniformly convex hypersurfaces in the hyperbolic space $\mathbb{H}^{n+1}\ (n\geqslant 2)$. We prove that if the initial closed hypersurface is smooth and uniformly convex, then the smooth solution exists for all positive time and converges smoothly and exponentially to a geodesic sphere centered at the origin. The key step is to prove the upper bound of Gauss curvature and that uniform convexity preserves along the flow. As an application, we provide a new proof for an Alexandrov-Fenchel inequality comparing $(n-2)$th quermassintegral and volume of convex domains in $\mathbb{H}^{n+1}\ (n\geqslant 2)$.
\end{abstract}

\maketitle



\section{Introduction}
    In this paper, we consider a Gauss curvature type flow for uniformly convex hypersurfaces in the hyperbolic space $\mathbb{H}^{n+1}\ (n\geqslant 2)$ and we also provide a new proof for a quermassintegral inequality for convex domains.
    
    Let $n\geqslant 2$. We view $\mathbb{H}^{n+1}=\mathbb{S}^n\times [0,+\infty)$ as a warped product manifold equipped with metric $$\overline{g}_{\mathbb{H}^{n+1}}=\mathrm{d}\rho^2+\phi^2(\rho)g_{\mathbb{S}^n},$$ where $\phi(\rho)=\sinh\rho,\ \rho\in[0,+\infty)$. Define $V=\phi(\rho)\partial_{\rho}$. It's well known that $V$ is a conformal Killing field. For a hypersurface $M\subset\mathbb{H}^{n+1}$, let $\nu$ be its unit outward normal vector. We denote $u=\langle V,\nu \rangle$ be its support function. Recall that, $M$ is star-shaped, if its support function is positive everywhere on $M$; while $M$ is uniformly convex (resp. convex) if its principal curvatures are all positive (resp. non-negative).
    
	Consider a smooth family of embeddings $X:\mathbb{S}^{n
	}\times [0,T)\to\mathbb{H}^{n+1}$ satisfying
	\begin{equation}
		\label{1.1}
		\begin{cases}
			\partial_t X(x,t) = \left( \phi'(\rho)-uK^{\frac{1}{n}} \right)\nu(x,t), \\ X(\cdot,0)=X_{0}(\cdot),
		\end{cases}
	\end{equation}
    where $K(\cdot,t)$ denotes the Gauss curvature of $M_t=X(\mathbb{S}^n,t)$. The volume of the enclosed domain $\Omega_t$ is non-decreasing while the $(n-2)$th quermassintegral $\mathscr{A}_{n-2}(\Omega_t)$ (see Section \ref{subsec-quer} for precise definition) is non-increasing along the flow (\ref{1.1}). We will prove the following long time existence and convergence results for the flow (\ref{1.1}) with uniformly convex initial hypersurfaces.
    
	\begin{theorem}
		\label{thm1.1}
		Let $M_{0}=X_{0}(\mathbb{S}^{n})$ be a smooth, closed, uniformly convex hypersurface in hyperbolic space $\mathbb{H}^{n+1}$ 
		containing the origin, then the flow \eqref{1.1} has a unique smooth uniformly convex solution $M_{t}$ for all time $t\in[0,+\infty)$. When $t\rightarrow +\infty$, $M_{t}$ converges smoothly to a unique standard smooth geodesic sphere centered at the origin, and the convergence of $M_{t}$ in any $C^{k}$-norm is exponential.
	\end{theorem}
	
	Gauss curvature type flows in hyperbolic space have been concentrated widely in recent years, and they have many applications in geometric problems. Andrews and Chen \cite{Andrews2017CurvatureFI} considered Gauss curvature type flow in $\mathbb{H}^3$
	\begin{equation*}
		\partial_t X(x,t) = - (K-1)\nu(x,t).
	\end{equation*}
	They proved that if the initial hypersurface has positive scalar curvature, then the solution surface is uniformly convex and converges to a point in finite time. After rescalings, the solution is asymptotic to a shrinking sphere in finite time. Chen and Huang \cite{flowgauss} considered the flow by $\alpha$-th power of the Gauss curvature in $\mathbb{H}^{n+1}$
	\begin{equation*}
		\partial_t X(x,t)= -K^{\alpha}\nu(x,t).
	\end{equation*}
	They proved that if the initial hypersurface is uniformly convex, then the solution hypersurface exists and converges to a point in finite time for $\alpha>0$, and converges to a geodesic sphere after rescaling for $\alpha>\frac{1}{n+2}$. Their work generalized the corresponding results in Euclidean space $\mathbb{R}^{n+1}$, see \cite{Andrews1999GaussCF,Andrews2000MotionOH,Andrews2016FlowBP,MR3765656}. In \cite{WYZ2023}, Wei, Yang and Zhou considered a non-local type volume preserving Gauss curvature flow in $\mathbb{H}^{n+1}$
	\begin{equation*}
		\partial_t X(x,t) = (\phi(t)-K^{\alpha})\nu(x,t),
	\end{equation*}
    where $\alpha>0$ and $\phi(t)=\dfrac{1}{|M_t|}\int_{M_t} K^{\alpha} \mathrm{d}\mu_t$. They confirmed that for a smooth closed uniformly convex initial hypersurface, the solution exists for all time $t\in[0,+\infty)$ and converges smoothly and exponentially to a geodesic sphere as $t\to+\infty$. Li and Zhang \cite{Li2023AFA} studied the prescribed Gaussian curvature problem in $\mathbb{H}^{n+1}$ by using the Gauss curvature flow method.
	
	Our flow (\ref{1.1}) is a locally constrained curvature flow. Note that by the standard theory of parabolic equation, the flow \eqref{1.1} exists for a short time if initial  hypersurface is uniformly convex. The proof of long time existence and convergence require a priori estimates. $C^0$ estimate is obtained by parametrizing the flow as a graph of radial function on $\mathbb{S}^n$ and using standard maximum principle. By \cite[Theorem 2.7.10]{curvatureproblem}, for a convex solution to the flow (\ref{1.1}), once we obtain the $C^0$ estimate, the $C^1$ estimate follows, and we conclude that the star-shapedness preserves along the flow \eqref{1.1}. Therefore, the main difficulty for the long time existence is the $C^2$ estimate. Inspired by the method in \cite[Section 4]{Li2023AFA} and \cite[Lemma 4.2]{alexandrovproblem2020}, we will choose appropriate auxiliary functions to prove the upper bound of Gauss curvature $K$ and that uniform convexity preserves along the flow (\ref{1.1}), from which we obtain the $C^2$ estimate. Then by the standard procedure, we obtain the long time existence. Note that by the arithmetic mean and geometric mean inequality, the volume of the enclosed domain is monotonically non-decreasing along the flow \eqref{1.1}. This property helps us to verify that the limit hypersurface is a geodesic sphere centered at the origin and to prove the smooth convergence. The exponential convergence follows by estimating the gradient of radial function along the flow \eqref{1.1} inspired by \cite[Section 8.3]{Locallyconstrained}.
	
	Locally constrained curvature flows play an important role in proving Alexandrov-Fenchel type inequalities, which ask whether
	\begin{equation}
		\mathscr{A}_k(\Omega) \geqslant \xi_{k,l}\left(\mathscr{A}_\ell(\Omega)\right)\ \ (-1\leqslant \ell<k\leqslant n-1)  \label{equ-alexandrov}
	\end{equation}
    holds for a smooth bounded domain $\Omega$ in space forms. Here $\mathscr{A}_k(\Omega)$ are quermassintegrals for the smooth bounded domain $\Omega$, see Section \ref{subsec-quer} for precise definition. $\xi_{k,l}(s)$ is the unique function such that equality in (\ref{equ-alexandrov}) holds if and only if $\Omega$ is a geodesic ball. Equality in (\ref{equ-alexandrov}) holds if and only if $\Omega$ is a geodesic ball.
    
    In \cite{Guan2009TheQI}, Guan and Li used inverse curvature flows to prove quermassintegral inequalities \eqref{equ-alexandrov} for $k$-convex (i.e. $\sigma_j(\kappa)\geqslant 0,\ j=1,2,\cdots,k$) and star-shaped hypersurfaces in $\mathbb{R}^{n+1}$ when $-1\leqslant \ell < k \leqslant n-1$. In \cite{Guan2013AMC}, Guan and Li studied the mean curvature type flow
	\begin{equation}
		\partial_t X= (n\phi'(\rho)-uH)\nu \label{equ-guanli}
	\end{equation}
    in space forms, where $H$ is the mean curvature of the evolving hypersurfaces. They proved that if the initial closed hypersurface is star-shaped, then the solution of (\ref{equ-guanli}) exists for all positive time and converges smoothly and exponentially to a geodesic sphere centered at the origin. They also proved that the uniform convexity preserves along the flow (\ref{equ-guanli}) in $\mathbb{R}^{n+1}$. Note that the volume preserves while area decreases along the flow (\ref{equ-guanli}), the isoperimetric inequality for star-shaped hypersurfaces in $\mathbb{R}^{n+1}$ follows. Later, Chen-Guan-Li-Scheuer \cite{chenguanlischeuer2022} proved that the uniform convexity preserves along the flow (\ref{equ-guanli}) in $\mathbb{S}^{n+1}$. Chen and Sun \cite{CHEN2022108203} proved quermassintegral inequalities \eqref{equ-alexandrov} for convex hypersurfaces in $\mathbb{S}^{n+1}$ when $\ell=k-2,\ 1\leqslant k\leqslant n-1$.
    
    In $\mathbb{H}^{n+1}$, Li, Wei and Xiong \cite{LWX2014} proved quermassintegral inequalities \eqref{equ-alexandrov} when $\ell=-1,\ k=2$ for star-shaped and two-convex hypersurfaces. Wang and Xia \cite{Wang-Xia2014} proved quermassintegral inequalities \eqref{equ-alexandrov} when $-1\leqslant \ell<k\leqslant n-1$ for h-convex (i.e. $\kappa_i\geqslant 1,\ i=1,\cdots,n$) domains. Later, Hu-Li-Wei \cite{Locallyconstrained} gave a new proof of \eqref{equ-alexandrov} for h-convex domains. Andrews, Chen and Wei \cite{Ben2021} proved \eqref{equ-alexandrov} when $-1=\ell<k\leqslant n-1$ for domains whose boundary has positive sectional curvatures. Brendle, Guan and Li \cite{BGL2018,Li2021IsoperimetricTI} proved quermassintegral inequalities \eqref{equ-alexandrov} when $-1\leqslant \ell<k=n-1$ for convex domains and $\ell=0,\ k=1$ for star-shaped and mean convex domains. Hu and Li \cite[Theorem 1.3]{HU2023108826} proved \eqref{equ-alexandrov} when $-1\leqslant \ell<k=n-2$ for uniformly convex domains.
	
	As far as we know, quermassintegral inequalities \eqref{equ-alexandrov} for convex domains in $\mathbb{H}^{n+1}$ remain open in many cases. Note that $(n-2)$th quermassintegral is non-increasing while the volume is non-decreasing along our flow \eqref{1.1}. As an application, we will use the convergence results of flow (\ref{1.1}) to provide a new proof for the Alexandrov-Fenchel inequality comparing $(n-2)$th quermassintegral and the volume of convex domains in $\mathbb{H}^{n+1}$.
	
	\begin{theorem}
		\label{thm-afineq}
		Let $M$ be a closed, smooth, convex hypersurface in $\mathbb{H}^{n+1}\ (n\geqslant 2)$ enclosing a domain $\Omega$, then 
		\begin{equation}
			\mathscr{A}_{n-2}(\Omega)\geqslant \xi_{n-2,-1}\left(\mathscr{A}_{-1}(\Omega)\right),  \label{equ-afine}
		\end{equation} 
		where $\xi_{n-2,-1}=\xi_{n-2}\circ\xi_{-1}^{-1}$, $\xi_k(r)=\mathscr{A}_k(B_r)$ is a monotone function, $B_r$ is a geodesic ball of radius $r$ centered at the origin, $\xi_k^{-1}$ is the inverse function of $\xi_k$. Equality holds if and only if $M$ is a geodesic sphere.
	\end{theorem}

	We will first prove the inequality \eqref{equ-afine} for uniformly convex hypersurfaces by using the convergence result Theorem \ref{thm1.1}. Then for a convex hypersurface, Theorem \ref{thm-afineq} follows by an approximation argument as in \cite{Guan2009TheQI} and \cite{WYZ2023}.
	
	The paper is organized as follows. In Section \ref{section2}, we will gather some preliminaries for our proof, including quermassintegrals, formulas for hypersurfaces in the hyperbolic space, the general evolution equations for curvature flows, and parametrization by radial graph. In Section \ref{section3}, we will prove $C^0$ and $C^1$ estimates. Especially, we conclude that star-shapedness preserves along the flow \eqref{1.1}. In Section \ref{section4}, we will prove the upper bound of Gauss curvature and the preservation of uniform convexity to establish the $C^2$ estimate. Then the long time existence follows by the standard procedure. In Section \ref{section5}, we will obtain the smooth and exponential convergence results and complete the proof of Theorem \ref{thm1.1}. In Section \ref{section6}, we will prove Theorem \ref{thm-afineq}, the quermassintegral inequality for convex domains.
	
	\begin{ack}
		The authors would like to thank Professor Yong Wei for his helpful discussions and support. This work is supported by National Key Research and
		Development Program of China 2021YFA1001800. 
	\end{ack}

\section{Preliminaries}
\label{section2}
    In this section, we will collect some preliminaries for our proof, including quermassintegrals, formulas for hypersurfaces in hyperbolic space, general evolution equations and parametrization by radial graph.
	
	\subsection{Hypersurfaces in hyperbolic space}
		We consider $\mathbb{H}^{n+1}=\mathbb{S}^n\times [0,\infty)$ as a warped product manifold equipped with metric
		\begin{equation}
			\overline{g}=\mathrm{d}\rho^2+\phi^2(\rho)g_{\mathbb{S}^n}, \nonumber
		\end{equation}
		where $\phi(\rho)=\sinh \rho$, $\rho\in[0,\infty)$, $g_{\mathbb{S}^n}$ denotes the standard metric on $\mathbb{S}^n$. Let
		\begin{equation}
			\Phi(\rho)=\int_0^\rho\phi(s)\mathrm{d}s=\cosh \rho-1.\nonumber
		\end{equation}
		It's well known that $V=D\Phi=\phi(\rho)\partial_{\rho}$  is a conformal Killing field on $\mathbb{H}^{n+1}$, that is
		\begin{equation}
			D V = \phi'(\rho)\overline{g}.\nonumber
		\end{equation}
	
		Let $M$ be a smooth closed hypersurface in $\mathbb{H}^{n+1}\ (n\geqslant 2)$ with induced metric $g$ and unit outward normal vector $\nu$. Let $\{x^1,x^2,\cdots,x^n\}$ be a local coordinate system on $M$, we denote $g_{ij}=g(\partial_i,\partial_j)$ and its second fundamental form $h_{ij}=\langle D_{\partial_i}\nu,\partial_j \rangle$. The Weingarten matrix is denoted by $\mathcal{W}=(h^i_j)$, where $h^i_j=g^{ik}h_{kj}$, and $(g^{ij})$ is the inverse matrix of $(g_{ij})$. The principal curvatures $\kappa=(\kappa_1,\kappa_2,\cdots,\kappa_n)$ are eigenvalues of the Weingarten matrix $\mathcal{W}$. The connection on $M$ is denoted by $\nabla$.
	
		The following formulas hold for smooth hypersurfaces in $\mathbb{H}^{n+1}$, see \cite[Lemmas 2.2 and 2.6]{Guan2013AMC}.
		\begin{lemma}\label{lem-hypersurface}
			Let $(M,g)$ be a smooth hypersurface in $\mathbb{H}^{n+1}$ with local frame $\{\partial_1,\cdots,\partial_n\}$, then $\Phi|_M$ satisfies
			\begin{equation}
				\nabla_i\Phi =\langle V,e_i \rangle,\ \ \nabla_j\nabla_i \Phi= \phi' g_{ij} - u h_{ij}, \label{Phi}
			\end{equation}
			and the support function $u=\langle V,\nu\rangle$ satisfies
			\begin{equation}
				\label{du}
				\nabla_i u = \langle V,e_l \rangle h_i^l,\ \ \nabla_j\nabla_i u= \langle V,\nabla h_{ij} \rangle + \phi' h_{ij} - u(h^2)_{ij},
			\end{equation}
			where $(h^2)_{ij}=h_{ik}h^k_j$.
		\end{lemma}
		Then we have the first and second derivatives of the distance function $\rho$.
		
		\begin{corollary}[{\cite[Corollary 2.1]{Li2023AFA}}]
			We have
			\begin{equation}
				\nabla_i \rho = \dfrac{\langle V,e_i \rangle}{\phi},\ \ \nabla_j\nabla_i \rho= \dfrac{\phi'}{\phi}(g_{ij}-\nabla_{i} \rho\nabla_{j} \rho)-\dfrac{uh_{ij}}{\phi}.\label{rho}
			\end{equation}
		\end{corollary}
		\begin{proof}
			Observe that
			\begin{equation}
				\nonumber
				\nabla_{i}\Phi=\phi\nabla_{i}\rho,\quad \nabla_j\nabla_i \Phi=\phi\nabla_j\nabla_i \rho+\phi'\nabla_{j}\rho\nabla_{i}\rho.
			\end{equation}
		Combining with (\ref{Phi}), we get (\ref{rho}) by a direct calculation.
		\end{proof}
	
	    The $k$th elementary symmetric polynomial of $\kappa$ is defined by
	    \begin{equation}
	    	\sigma_k(\kappa) = \sum\limits_{1\leqslant i_1<i_2<\cdots<i_k\leqslant n}\kappa_{i_1}\kappa_{i_2}\cdots\kappa_{i_k}.\nonumber
	    \end{equation}
	    We also have the convention that $\sigma_0=1$ and $\sigma_k=0$ for $k>n$. This includes  the mean curvature $H=\sigma_1(\kappa)$ and Gauss curvature $K=\sigma_n(\kappa)$ as special cases. We have the following Minkowski formulas.
		\begin{lemma}[{\cite[Proposition 2.5]{Guan2013AMC}}]
			Let $M$ be a smooth closed hypersurface in $\mathbb{H}^{n+1}$. Then
			\begin{equation}
				(n-k)\int_M \phi' \sigma_k(\kappa) \mathrm{d}\mu = (k+1)\int_M u \sigma_{k+1}(\kappa) \mathrm{d}\mu,\ \ k=0,1,\cdots,n-1, \label{equ-minkf}
			\end{equation}
			where $\sigma_k(\kappa)$ is the $k$th elementary symmetric polynomial of $\kappa$.
		\end{lemma}
	
	\subsection{Quermassintegrals}
	\label{subsec-quer}
	Recall that for a convex domain $\Omega$ in hyperbolic space $\mathbb{H}^{n+1}$, the quermassintegrals of $\Omega$ are defined as follows (see \cite[Definition 2.1]{G-Solanes05}):
	\begin{equation}
		\mathscr{A}_k(\Omega) = (n-k){n\choose k} \dfrac{\omega_{k}\cdots\omega_0}{\omega_{n-1}\cdots\omega_{n-k-1}}\int_{\mathcal{L}_{k+1}} \chi(L_{k+1}\cap \Omega) \mathrm{d}L_{k+1},\ k=0,1,\cdots,n-1,\nonumber
	\end{equation}
	where $\omega_k=|\mathbb{S}^k|$ denotes the area of $k$-dimensional sphere, $\mathcal{L}_{k+1}$ is the space of $k+1$-dimensional totally geodesic subspaces $L_{k+1}$ in $\mathbb{H}^{n+1}$. The function $\chi$ is defined to be 1 if $L_{k+1}\cap \Omega\neq\varnothing$ and to be 0 otherwise. Furthermore, we set
	\begin{equation}
		\mathscr{A}_{-1}(\Omega) = |\Omega|,\ \ 	\mathscr{A}_0(\Omega)=|\partial\Omega|,\ \ \mathscr{A}_n(\Omega) = |\mathbb{B}^{n+1}| = \dfrac{\omega_n}{n+1}.\nonumber\nonumber
	\end{equation}
	
	 The quermassintegrals and the curvature integrals of a smooth domain $\Omega$ in $\mathbb{H}^{n+1}$ are related by the following equations:
	\begin{eqnarray}
		&&\mathscr{A}_1(\Omega) = \int_{\partial\Omega} \sigma_1(\kappa) \mathrm{d}\mu_g - n\mathscr{A}_{-1}(\Omega),\notag\\
		&&\mathscr{A}_k(\Omega) = \int_{\partial\Omega} \sigma_k(\kappa) \mathrm{d}\mu_g - \dfrac{n-k+1}{k-1}\mathscr{A}_{k-2}(\Omega),\ \ k=2,\cdots,n.
	\end{eqnarray}
	The quermassintegrals for smooth domains satisfy a nice variation property:
	\begin{eqnarray}
		&&\dfrac{\mathrm{d}}{\mathrm{d} t}\mathscr{A}_{-1}(\Omega_t) = \int_{M_t} f \mathrm{d}\mu_t,\\
		&&\dfrac{\mathrm{d}}{\mathrm{d} t}\mathscr{A}_k(\Omega_t) = (k+1) \int_{M_t} f\sigma_{k+1}(\kappa) \mathrm{d}\mu_t,\ \ k=0,1,\cdots,n-1
	\end{eqnarray}
	along any outward normal variation with speed $f$.
	
	\subsection{Evolution equations}
	\indent 
	For convenience, we rewrite the flow \eqref{1.1} as
		\begin{equation}
			\partial_tX(x,t)=(\phi'-\Theta)\nu(x,t), \label{flow}
		\end{equation}
	where $\Theta=uK^{\frac{1}{n}}$. The following lemma can be found in \cite[Lemma 3.1]{Guan2013AMC} and \cite[Lemma 2.2]{Li2023AFA}. We include a proof here for reader's convenience.
	
		\begin{lemma}
			Along the flow \eqref{flow}, we have the following evolution equations. The induced metric evolves by
			\begin{equation}
				\partial_{t}g_{ij}=-2\Theta h_{ij}+2\phi'h_{ij},\label{gij}
			\end{equation}
		The support function evolves by
			\begin{equation}
				\partial_{t}u=-\phi'\Theta+(\phi')^{2}-\phi^{2}|\nabla\rho|^{2}+\langle V,\nabla\Theta \rangle,\label{u}
			\end{equation}
		The Weingarten matrix evolves by
			\begin{equation}
				\partial_{t}h_{i}^{j}=\nabla_{i}\nabla^{j}\Theta+\Theta h_{i}^{k}h_{k}^{j}-\phi'h_{i}^{k}h_{k}^{j}+uh_{i}^{j}-\Theta\delta_{i}^{j},\label{hij}
			\end{equation}
		where $\nabla$ is the Levi-Civita connection of the induced metric on $M_{t}$.
		\end{lemma}
		\begin{proof}
			By a direct calculation, we have
			\begin{equation}
				\nonumber
				\begin{aligned}
					\partial_{t}g_{ij}=\partial_{t}\langle\partial_{i}X,\partial_{j}X \rangle &= \langle D_{i}((\phi'-\Theta)\nu(x,t)),\partial_{j}X \rangle + \langle \partial_{i}X ,D_{j}((\phi'-\Theta)\nu(x,t)) \rangle \\
					&=(\phi'-\Theta)( \langle D_{i}\nu , \partial_{j}X \rangle +\langle \partial_{i}X , D_{j}\nu \rangle ) =-2\Theta h_{ij}+2\phi' h_{ij}.
				\end{aligned}
			\end{equation}
		Since $\partial_{t}\nu$ is tangential,
			\begin{equation}
				\label{nu}
				\begin{aligned}
					\partial_{t}\nu =\langle \partial_{t}\nu,\partial_{i}X \rangle g^{il}\partial_{l}X 
					&=-\langle \nu,\partial_{i}((\phi'-\Theta)\nu) \rangle g^{il}\partial_{l}X \\
					&=\partial_{i}(-\phi'+\Theta)g^{il}\partial_{l}X\\
					&=\nabla(-\phi'+\Theta)\\
					&=\nabla\Theta-\phi\nabla\rho.
				\end{aligned}	
			\end{equation}
		Using (\ref{nu}), we obtain the evolution of the support function $u$ as follows:
			\begin{equation}
				\nonumber
				\begin{aligned}
					\partial_{t}u=\partial_{t} \langle V,\nu \rangle 
						&=\phi' \langle (\phi'-\Theta)\nu,\nu \rangle + \langle V, \nabla\Theta-\phi\nabla\rho \rangle \\
						&=-\phi'\Theta+(\phi')^{2}-\phi^{2}|\nabla\rho|^{2}+\langle V,\nabla\Theta \rangle .
				\end{aligned}
			\end{equation}
		Now we calculate the evolution of $h_{ij}$ as
			\begin{equation}
				\nonumber
				\begin{aligned}
					\partial_{t}h_{ij}
						&=-\partial_{t}\langle D_{\partial_{i}X}\partial_{j}X,\nu \rangle \\
						&=-\langle D_{\partial_{i}X}D_{\partial_{j}X}((\phi'-\Theta)\nu),\nu \rangle - R^{\mathbb{H}^{n+1}}(\partial_{i}X,\partial_{t}X,\partial_{j}X,\nu)-\langle D_{\partial_{i}X}\partial_{j}X,\nabla\Theta-\nabla\phi' \rangle \\
						&=\partial_i \partial_j(\Theta-\phi')-(\Theta-\phi')(h^{2})_{ij}+(\phi'-\Theta)g_{ij}-\langle D_{\partial_{i}X}\partial_{j}X,\nabla\Theta-\nabla\phi' \rangle \\
						&=\nabla_{i}\nabla_{j}(\Theta-\phi')-(\Theta-\phi')h_{i}^{k}h_{kj}+(\phi'-\Theta)g_{ij}\\
						&=\nabla_{i}\nabla_{j}\Theta-\Theta h_{i}^{k}h_{kj}+\phi' h_{i}^{k}h_{kj}-\Theta g_{ij}+uh_{ij},
				\end{aligned}
			\end{equation}
		where we used (\ref{Phi}) in the last equality. From (\ref{gij}), we have
			\begin{equation}
				\nonumber
				\partial_{t}g^{ij}=-g^{il}(\partial_{t}g_{lm})g^{mj}=2\Theta g^{il}h_{l}^{j}-2\phi'g^{il}h_{l}^{j}.
			\end{equation}
		Thus
			\begin{equation}
				\nonumber
				\partial_{t}h_{i}^{j}=\partial_{t}h_{il}g^{lj}+h_{il}\partial_{t}g^{lj}=\nabla_{i}\nabla^{j}\Theta+\Theta h_{i}^{k}h_{k}^{j}-\phi' h_{i}^{k}h_{k}^{j}+uh_{i}^{j}-\Theta\delta_{i}^{j}.
			\end{equation}
		\end{proof}
	\subsection{Parametrization by radial graph}
	For a closed star-shaped hypersurface $M \subset \mathbb{H}^{n+1}$, we can parametrize it as a graph of the radial function $\rho(\theta):\mathbb{S}^{n}\rightarrow\mathbb{R}^+$, i.e.,
		\begin{equation*}
			M^{n}=\{(\rho(\theta),\theta) \mid \rho:\mathbb{S}^{n}\rightarrow\mathbb{R}^+,\theta\in\mathbb{S}^{n}\},
		\end{equation*}
	where $\theta = (\theta_{1},\cdots, \theta_{n})$ is a local normal coordinate system of $\mathbb{S}^{n}$ and $\rho$ is a smooth function on  $\mathbb{S}^{n}$. Let $ f_{i}=\overline{\nabla}_if$, $f_{ij} =\overline{\nabla}^{2}_{ij}f$, where $\overline{\nabla}$ is the Levi-Civita connection on $\mathbb{S}^{n}$ with respect to the standard metric $g_{\mathbb{S}^{n}}$.
	
	 The tangent space of $M^{n}$ is spanned by $X_i = \rho_i \partial_{\rho} + \partial_{\theta_i}$ (see \cite[Section 2.3]{Locallyconstrained}) and the unit outward normal vector is
	 	\begin{equation*}
	 		\nu=\dfrac{\partial_{\rho}-\frac{\rho^{i}\partial_{\theta_i}}{\phi^{2}}}{w},
	 	\end{equation*}
 	where we set
 		\begin{equation}
 			\label{w}
 			w=\sqrt{1+\frac{|\overline{\nabla}\rho|^{2}}{\phi^{2}}}.
 		\end{equation}
 	Then the support function and the induced metric can be expressed as
 		\begin{equation}
 			\label{uw}
 			u=\dfrac{\phi^{2}}{\sqrt{\phi^{2}+|\overline{\nabla}\rho|^{2}}}=\dfrac{\phi}{w},
 		\end{equation}
 		\begin{equation}
 			g_{ij}=\phi^{2}\delta_{ij}+\rho_i \rho_j,\quad g^{ij}=\dfrac{1}{\phi^2}(\delta^{ij}-\dfrac{\rho_i \rho_j}{\phi^{2}+|\overline{\nabla}\rho|^{2}}).
 		\end{equation}
 	The second fundamental form is given by
 		\begin{equation}
 			h_{ij}=\dfrac{-\phi\rho_{ij}+2\phi'\rho_i \rho_j+\phi^{2}\phi'\delta_{ij}}{\sqrt{\phi^{2}+|\overline{\nabla}\rho|^{2}}},
 		\end{equation}
 	and we have the Weingarten matrix
 		\begin{equation}
 			\label{hi^j}
 			h_i^j=\dfrac{1}{\phi^2\sqrt{\phi^{2}+|\overline{\nabla}\rho|^{2}}}(\delta^{jk}-\dfrac{\rho_j \rho_k}{\phi^{2}+|\overline{\nabla}\rho|^{2}})(-\phi\rho_{ki}+2\phi'\rho_k \rho_i+\phi^{2}\phi'\delta_{ki}).
 		\end{equation}
 	Similar to \cite{Guan2013AMC}, the flow (\ref{flow}) can be written as a scalar parabolic PDE for the radial function $\rho$:
 		\begin{equation}
 			\label{equ-rho}
 			\begin{cases}
 				\partial_t \rho(\theta,t) = -\phi K^{\frac{1}{n}}+\phi' w,\quad \mbox{for}\  (\theta,t)\in\mathbb{S}^n\times [0,+\infty), \\ \rho(\cdot,0)=\rho_{0}(\cdot),
 			\end{cases} 			
 		\end{equation}
 	where $w$ is the function defined in (\ref{w}).
 	
 	Let $\gamma(\rho)$ satisfy
 	\begin{equation}
 		\dfrac{\mathrm{d}\gamma}{\mathrm{d}\rho} = \dfrac{1}{\phi(\rho)},  \label{equ-gammap}
 	\end{equation}
    then the flow (\ref{flow}) can be equivalent to the scalar parabolic equation for $\gamma$ on $\mathbb{S}^n$:
    \begin{equation}
    	\partial_t \gamma = \dfrac{\partial_t \rho}{\phi(\rho)} =  \dfrac{\phi'}{\phi}w-K^{\frac{1}{n}},\ \ (\theta,t)\in\mathbb{S}^n\times [0,+\infty).  \label{equ-gammat}
    \end{equation}

\section{$C^{0}$ and $C^{1}$ estimates}
\label{section3}
In this section, we will establish the $C^{0}$ and $C^{1}$ estimates of the flow (\ref{1.1}) for the proof of Theorem \ref{thm1.1}. Especially, we
show that the flow hypersurface $M_{t}$ preserves star-shapedness along (\ref{1.1}).

	\subsection{$C^{0}$ Estimate}

	We first use the maximum principle to show that the radial function $\rho$ of (\ref{equ-rho}) has uniform bounds.
		\begin{lemma}\label{lemma3.1}
			
			Let $\rho(\cdot,t)$ be a smooth, positive, uniformly convex solution to \eqref{equ-rho} on $\mathbb{S}^{n}\times [0,T)$.
   			 Then
    		\begin{equation}
    			\underset{\mathbb{S}^{n}}{\min}\rho(\cdot,0)\le\rho(\cdot,t)\le \underset{\mathbb{S}^{n}}{\max}\rho(\cdot,0), \qquad \forall\ t\in[0,T). \label{equ-c}
  		  	\end{equation}
		\end{lemma}

		\begin{proof}
			Fix time $t$ and suppose that $\rho$ attains its spatial maximum at point $(p_{0}, t)$. At $(p_{0}, t)$, we have $|\overline{\nabla}\rho|=0$ and $(\rho_{ij})\le0$. From (\ref{hi^j}),
			\begin{equation}
			h_{i}^{j}=\dfrac{\phi'}{\phi}\delta_{i}^{j}-\dfrac{1}{\phi^{2}}\rho_{i}^{j}
			\ge\dfrac{\phi'}{\phi}\delta_{i}^{j}.\label{3.2}
			\end{equation}
			Hence we have $K^{\frac{1}{n}}\ge\dfrac{\phi'}{\phi}$. Substituting it into (\ref{equ-rho}), we obtain
			\begin{equation}
				\partial_{t}\rho_{\max}\le-\phi\dfrac{\phi'}{\phi}+\phi'=0. \label{3.3}
			\end{equation}
			This implies that $\underset{\mathbb{S}^{n}}{\max}\rho(\cdot,t)$ is non-increasing. In particular,
			\begin{equation}
				\rho(\cdot,t)\le \underset{\mathbb{S}^{n}}{\max}\rho(\cdot,0), \qquad \forall\ t\in[0,T). \nonumber
			\end{equation}
			Similarly, $\underset{\mathbb{S}^{n}}{\min}\rho(\cdot,t)$ is non-decreasing, and
			\begin{equation}
				\rho(\cdot,t)\ge \underset{\mathbb{S}^{n}}{\min}\rho(\cdot,0), \qquad \forall\ t\in[0,T). \nonumber
			\end{equation}
			This proves the lemma.
		\end{proof}
	
	\subsection{$C^{1}$ Estimate}
	
		We now give the uniform upper bound of the gradient of $\rho$. It is obtained by combining \cite[Theorem 2.7.10]{curvatureproblem} and the $C^0$ estimate (Lemma \ref{lemma3.1}).
		\begin{lemma}\label{lemma3.3}
			Let $\rho(\cdot,t)$ be a smooth, positive, uniformly convex solution to \eqref{equ-rho} on $\mathbb{S}^{n}\times [0,T)$. Then 
			\begin{equation}
				|\overline{\nabla}\rho|\le C,  \label{equ-c0}
			\end{equation}
			where $C$ depends only on the initial hypersurface.
		\end{lemma}
	    \begin{proof}
	    	Since $\rho(\cdot,t)$ is a smooth uniformly convex solution to \eqref{equ-rho} on $\mathbb{S}^n\times [0,T)$, we have $|\overline{\nabla}\rho|(\cdot,t)\leqslant C\left(\max_{\mathbb{S}^n}\rho(\cdot,t)-\min_{\mathbb{S}^n}\rho(\cdot,t)\right)$ by \cite[Theorem 2.7.10]{curvatureproblem}. Combining the $C^0$ estimate \eqref{equ-c}, we obtain \eqref{equ-c0}, where the positive constant $C$ depends only on the initial hypersurface.
	    \end{proof}
		
		Along the flow \eqref{flow}, star-shapedness preservation follows by combining Lemmas \ref{lemma3.1} and \ref{lemma3.3}.
		
		\begin{corollary}
			\label{cor-starshapedness}
			Along the flow \eqref{flow}, the uniformly convex hypersurface $M_{t}$ preserves star-shapedness and the support function $u$ satisfies
			\begin{equation}
				\frac{1}{C}\le u \le C
			\end{equation}
			for all $t\in [0, T)$ for some constant $C>0$, where $C$ only depends on the initial hypersurface.
		\end{corollary}
		\begin{proof}
			Recall (\ref{uw}),
			\begin{equation}
			u=\dfrac{\phi}{\sqrt{1+\frac{|\overline{\nabla}\rho|^{2}}{\phi^{2}}}}.\nonumber
			\end{equation}
			The upper and lower bounds of $u$ follow from Lemmas \ref{lemma3.1} and \ref{lemma3.3}. Besides, we have
			\begin{equation}
				\langle\partial_{\rho},\nu\rangle=\dfrac{u}{\phi}=\dfrac{1}{w}=\dfrac{1}{\sqrt{1+\frac{|\overline{\nabla}\rho|^{2}}{\phi^{2}}}}\ge\dfrac{1}{C^{'}}\nonumber
			\end{equation}
		for some $C'>0$ depends on $\max_{\mathbb{S}^n\times [0,T)}|\overline{\nabla}\rho|$ and $\min_{\mathbb{S}^n\times [0,T)}\rho$. Thus the hypersurface $M_t$ preserves star-shapedness along the flow (\ref{flow}).
		\end{proof}
	
\section{$C^{2}$ estimate}
\label{section4}
In this section, we will prove the $C^{2}$ estimate of the flow (\ref{flow}).

	\subsection{The upper bound of Gauss curvature $K$}
	We show that $K$ is bounded above along (\ref{flow}).
		\begin{lemma}\label{lemma4.1}
			Let $X(\cdot,t)$ be a uniformly convex solution to the flow \eqref{flow} which encloses the origin for $t\in[0,T)$. Then there is a positive constant C depending only on n, $M_{0}$ and the uniform bounds of $\rho$ such that
			\begin{equation}
				K(\cdot,t)\le C,\qquad\forall\ t\in [0,T).
			\end{equation}
		\end{lemma}
		
		\begin{proof}
			Consider the auxiliary function
			\begin{equation}
				Q=\log\Theta-\log(u-a),\nonumber
			\end{equation}
			where $\Theta=uK^{\frac{1}{n}}$ and $a=\frac{1}{2}\inf_{M\times[0,T)}u>0$.
			
			 Using (\ref{u}) and (\ref{hij}), we then have
			 \begin{equation}\label{4.2}
			 	\begin{aligned}
		 		   \partial_{t}\Theta=
		 		   &\partial_{t}uK^{\frac{1}{n}}+\frac{1}{n}\dfrac{\Theta}{K}\dfrac{\partial K}{\partial h_{i}^{j}}\left(\nabla_{i}\nabla^{j}\Theta + \Theta h_{i}^{k}h_{k}^{j}-\phi' h_{i}^{k}h_{k}^{j}+uh_{i}^{j}-\Theta\delta_{i}^{j}\right) \\
		 		   =&	[(\phi')^{2}-\phi^{2}|\nabla\rho|^{2}-\phi'\Theta+\langle V,\nabla\Theta \rangle]K^{\frac{1}{n}} \\
		 		   &+\frac{1}{n}\dfrac{\Theta}{K}\dot{K}^{ij}\Theta_{ij}+\dfrac{\Theta^{2}}{n}H
		 		   -\dfrac{\Theta\phi'}{n}H	+ u\Theta -   \dfrac{\Theta^{2}}{n}\dfrac{\sigma_{n-1}}{K},
			 	\end{aligned}
			 \end{equation}
		 where we used
		 	\begin{equation*}
		 		\dfrac{\partial K}{\partial h_{i}^{j}}\delta_{i}^{j}=\sigma_{n-1} \quad \mbox{and} \quad \dfrac{\partial K}{\partial h_{i}^{j}}h_{i}^{k}h_{k}^{j}=KH
		 	\end{equation*}
	 	in the last equality and denote $\dot{K}^{ij}=\dfrac{\partial K}{\partial h_i^j}$. Recall (\ref{u}),
	 		\begin{equation}
	 			\begin{aligned}
	 				\partial_{t}u=&(\phi')^{2}-\phi^{2}|\nabla\rho|^{2}-\phi'\Theta+\langle V,\nabla\Theta \rangle\\
	 							=&(\phi')^{2}-\phi^{2}|\nabla\rho|^{2}-\phi'\Theta+K^{\frac{1}{n}}\langle V,\nabla u\rangle+\dfrac{uK^{\frac{1}{n}-1}}{n}\langle V,\nabla K \rangle\\
	 							=&(\phi')^{2}-\phi^{2}|\nabla\rho|^{2}-\phi'\Theta+K^{\frac{1}{n}}\langle V,\nabla u\rangle+\dfrac{\Theta}{nK}\dot{K}^{ij}\langle V,\nabla h_{ij}\rangle.\nonumber
	 			\end{aligned}
	 		\end{equation}
 		Combining with (\ref{du}), we obtain the evolution equation of $u$:
 			\begin{equation}\label{4.3}
 				\partial_{t}u=(\phi')^{2}-\phi^{2}|\nabla\rho|^{2}-2\phi'\Theta+K^{\frac{1}{n}}\langle V,\nabla u\rangle+\dfrac{u\Theta}{n}H+\dfrac{\Theta}{nK}\dot{K}^{ij}u_{ij}.
 			\end{equation}
 		At the spatial maximum point $(p_{0}, t)$ of $Q$ on $M_{t}$, we have
 			\begin{equation}\label{4.4}
 				\dfrac{\nabla\Theta}{\Theta}=\dfrac{\nabla u}{u-a}\quad \mbox{and} \quad
 				Q_{ij}=\dfrac{\Theta_{ij}}{\Theta}-\dfrac{u_{ij}}{u-a}\le 0.
 			\end{equation}
		Combining with (\ref{4.2}), (\ref{4.3}) and (\ref{4.4}), we obtain
			\begin{equation}
				\begin{aligned}
					\partial_{t}Q=&\dfrac{\partial_{t}\Theta}{\Theta}-\dfrac{\partial_{t}u}{u-a} \\
								 =&\dfrac{((\phi')^{2}-\phi^{2}|\nabla\rho|^{2})}{u}-\dfrac{\phi'\Theta}{u}+\dfrac{\langle V,\nabla\Theta\rangle}{u}+\dfrac{1}{nK}\dot{K}^{ij}\Theta_{ij}+\dfrac{1}{n}\Theta H-\dfrac{\phi'}{n}H+ u -\dfrac{\Theta}{n}\dfrac{\sigma_{n-1}}{K} \\
								 &-\dfrac{((\phi')^{2}-\phi^{2}|\nabla\rho|^{2})}{u-a}+2\dfrac{\phi'\Theta}{u-a}-\dfrac{u}{n(u-a)}\Theta H-\dfrac{\Theta}{n(u-a)K}\dot{K}^{ij}u_{ij}-\dfrac{\Theta}{u(u-a)}\langle V,\nabla u\rangle\\
								 \le&\dfrac{\Theta}{nK}Q_{ij}-\dfrac{a}{u(u-a)}((\phi')^{2}-\phi^{2}|\nabla\rho|^{2})+ u +\left(\dfrac{2\phi'}{u-a}-\dfrac{\phi'}{u}\right)\Theta-\dfrac{a}{n(u-a)}\Theta H, \nonumber
				\end{aligned}				
			\end{equation}
		where we used (\ref{4.4}) in the last inequality. Then we use the arithmetic mean and geometric mean inequality $H\ge nK^{\frac{1}{n}}$ and the $C^{0}$ estimate to obtain
			\begin{equation}
				\partial_{t}Q\le C_{1}+C_{2}\Theta-C_{3}\Theta^{2}\nonumber
			\end{equation}
		at the spatial maximum point $(p_0,t)$ for some $C_{1},C_{2},C_{3}>0$ depending only on  $n$, $M_{0}$ and the uniform bounds of $\rho$. Besides, there is
		a constant $C_{4}$ depending on the uniform upper and lower bounds of $u$ such that
			\begin{equation*}
				\dfrac{1}{C_{4}}\mathrm{e}^{Q}\le\Theta=(u-a)\mathrm{e}^{Q}\le C_{4}\mathrm{e}^{Q}.
			\end{equation*}
		We finally obtain that
			\begin{equation}
				\partial_{t}Q\le C_{1}+C_{2}C_{4}\mathrm{e}^{Q}-C_{3}(C_{4})^{-2}\mathrm{e}^{2 Q}\nonumber
			\end{equation}
		holds at the spatial maximum point $(p_0,t)$. Therefore, $Q\le \max\{C_{5}, Q_{\max}(0)\}$, where $C_{5}$ is a positive constant depending on $C_{1},C_{2},C_{3}$ and $C_{4}$. Hence $K=(\frac{u-a}{u}\mathrm{e}^{Q})^{n}$ is bounded from above by some positive constant $C$ depending only on  $n$, $M_{0}$ and the uniform bounds of $\rho$.
		\end{proof}

	\subsection{The bound of principal curvatures}
		We show that the principal curvatures of $M_{t}$ are uniformly bounded along the flow (\ref{flow}).
		
		\begin{lemma}\label{lemma4.2}
			Let $X(\cdot,t)$ be a uniformly convex solution to the flow \eqref{flow} which encloses the origin for $t\in[0,T)$. Then there is a positive constant $C$ depending only on n, $M_{0}$ and the uniform bounds of $\rho$ such that the principal curvatures satisfy
			\begin{equation}
				\dfrac{1}{C}\le\kappa_{i}\le C
			\end{equation}
			for all $t\in[0,T)$ and $i=1,2,\dots,n$.
		\end{lemma}
	
		\begin{proof}
			Denote by $\lambda(x,t)$ the maximal principal radii at $X(x,t)$. Let $A$ and $B$ be the positive constants to be determined later. Let $\tilde{Q}=\log\lambda(x,t)-Au+B\rho$. Fix an arbitrary time $T_{0}\in[0,T)$. Assume that $\tilde{Q}$ attains its maximum
			at $(x_{0},t_{0})$ provided $t\in[0,T_{0}]$. We then introduce a normal coordinate system $\{\partial_{i}\}$ around $(x_{0},t_{0})$ such that
			$\nabla_{\partial_{i}X}\partial_{j}X(x_{0},t_{0})=0$ for all $i,j= 1, 2,\dots, n$ and $h_{ij}(x_{0},t_{0})=\kappa_{i}(x_{0},t_{0})\delta_{ij}$. Furthermore, we can choose $\partial_{1}(x_{0},t_{0})$ as the
			eigenvector with respect to $\lambda(x_{0},t_{0})$, i.e., $\lambda(x_{0},t_{0})=h^{11}(x_{0},t_{0})$. Assume that $\{h^{ij}\}$ is the inverse matrix of $\{h_{ij}\}$. Clearly,
				\begin{equation}
					\lambda(x,t)=\max\{h^{ij}(x,t)\xi_{i}\xi_{j}\, |\, g^{ij}(x,t)\xi_{i}\xi_{j}=1\}.\label{4.5}
				\end{equation}
			For the continuity, using this coordinate system, we consider the auxiliary function
			\begin{equation}
				Q=\log v-Au+B\rho,
			\end{equation}
			where $v=\dfrac{h^{11}}{g^{11}}$. Note that $v(x,t) \le \lambda(x, t)$ from (\ref{4.5}) and $v(x_{0},t_{0}) = \lambda(x_{0},t_{0})$. Thus $Q(x, t) \le \tilde{Q}(x_{0},t_{0})$ for all $t\in [0, T_{0}]$.
			
			Now we calculate the derivatives of $v$ at $(x_{0},t_{0})$ as follows:
				\begin{equation}
					\begin{aligned}
						\partial_{t}v&=-(h^{11})^{2}\partial_{t} h_{11}+h^{11}\partial_{t} g_{11}=-(h^{11})^{2}\partial_{t}h_{1}^{1},\\
						\nabla_{i}v&=\dfrac{1}{g^{11}}\dfrac{\partial h^{11}}{\partial h_{pq}}\nabla_{i}h_{pq}=-\dfrac{1}{g^{11}}h^{1p}h^{1q}\nabla_{i}h_{pq},\quad 
						\nabla_{i}v(x_{0},t_{0})=-(h^{11})^{2}\nabla_{i}h_{11}.\label{4.6}
					\end{aligned}
				\end{equation}
			Then
				\begin{equation}
					\begin{aligned}
						\nabla_{j}\nabla_{i}v(x_{0},t_{0})
									&=\nabla_{j}(-\dfrac{1}{g^{11}}h^{1p}h^{1q}\nabla_{i}h_{pq})\\
									&=-h^{1q}\nabla_{j}h^{1p}\nabla_{i}h_{pq}-h^{1p}\nabla_{j}h^{1q}\nabla_{i}h_{pq}-h^{1p}h^{1q}\nabla_{j}\nabla_{i}h_{pq}\\
									&=h^{1q}h^{1r}h^{ps}\nabla_{j}h_{rs}\nabla_{i}h_{pq}+h^{1p}h^{1r}h^{qs}\nabla_{j}h_{rs}\nabla_{i}h_{pq}-h^{1p}h^{1q}\nabla_{j}\nabla_{i}h_{pq}\\
									&=-(h^{11})^{2}\nabla_{j}\nabla_{i}h_{11}+2(h^{11})^{2}h^{pp}\nabla_{i}h_{1p}\nabla_{j}h_{1p}.
					\end{aligned}		\nonumber			
				\end{equation}
			
			Now we calculate the first term in (\ref{hij}) as
				\begin{equation}
					\begin{aligned}
						\nabla_{j}\nabla_{i}\Theta
							&=u\nabla_{j}\nabla_{i}K^{\frac{1}{n}}+K^{\frac{1}{n}}\nabla_{j}\nabla_{i}u+\nabla_{j}u\nabla_{i}K^{\frac{1}{n}}+\nabla_{j}K^{\frac{1}{n}}\nabla_{i}u\\
							&=u\nabla_{j}(\dfrac{1}{n}K^{\frac{1}{n}-1}\dot{K}^{pq}\nabla_{i}h_{pq})+K^{\frac{1}{n}}\nabla_{j}\nabla_{i}u+\dfrac{1}{n}K^{\frac{1}{n}-1}\dot{K}^{pq}\nabla_{i}h_{pq}\nabla_{j}u+\dfrac{1}{n}K^{\frac{1}{n}-1}\dot{K}^{pq}\nabla_{j}h_{pq}\nabla_{i}u\\
							&=\dfrac{\Theta}{nK}\dot{K}^{pq}\nabla_{j}\nabla_{i}h_{pq}+\dfrac{1}{n}(\dfrac{1}{n}-1)\dfrac{\Theta}{K^{2}}\dot{K}^{pq}\dot{K}^{rs}\nabla_{j}h_{rs}\nabla_{i}h_{pq}+\dfrac{\Theta}{nK}\ddot{K}^{pq,rs}\nabla_{j}h_{rs}\nabla_{i}h_{pq}\\
							&\qquad+\dfrac{\Theta}{u}\nabla_{j}\nabla_{i}u+\dfrac{\Theta}{nuK}\dot{K}^{pq}\nabla_{i}h_{pq}\nabla_{j}u+\dfrac{\Theta}{nuK}\dot{K}^{pq}\nabla_{j}h_{pq}\nabla_{i}u. \nonumber
					\end{aligned}
				\end{equation}
			Due to Codazzi equation and Ricci identity, we have
				\begin{equation}
					\begin{aligned}
						\dot{K}^{pq}\nabla_{j}\nabla_{i}h_{pq}
								&=\dot{K}^{pq}\nabla_{j}\nabla_{q}h_{ip}=\dot{K}^{pq}(\nabla_{q}\nabla_{j}h_{ip}+h_{lp}R_{liqj}+h_{li}R_{lpqj})\\
								&=\dot{K}^{pq}(\nabla_{q}\nabla_{p}h_{ij}+h_{lp}h_{lq}h_{ij}-h_{lp}h_{lj}h_{iq}+h_{li}h_{lq}h_{pj}-h_{li}h_{lj}h_{pq}\\
								&\qquad\qquad-h_{lp}\delta_{lq}\delta_{ij}+h_{lp}\delta_{lj}\delta_{iq}-h_{li}\delta_{lq}\delta_{pj}+h_{li}\delta_{lj}\delta_{pq})\\
								&=\dot{K}^{pq}\nabla_{q}\nabla_{p}h_{ij}+KHh_{ij}-nKh_{il}h_{jl}-nK\delta_{ij}+\dot{K}^{pp}h_{ij}. \nonumber
					\end{aligned}
				\end{equation}
			It is direct to calculate that
				\begin{equation}
					\ddot{K}^{pq,rs}\nabla_{j}h_{rs}\nabla_{i}h_{pq}
							=\dfrac{\partial \left(Kh^{pq}\right)}{\partial h_{rs}}\nabla_{j}h_{rs}\nabla_{i}h_{pq}
							=\dfrac{\nabla_{j}K\nabla_{i}K}{K}-Kh^{pp}h^{qq}\nabla_{j}h_{pq}\nabla_{i}h_{pq}, \nonumber
				\end{equation}
			where we used $\dot{K}^{pq}=Kh^{pq}$ and $\dfrac{\partial h^{pq}}{\partial h_{rs}}=-h^{pr}h^{qs}$. Hence, at $(x_{0},t_{0})$,
				\begin{equation}
					\begin{aligned}
						\nabla_{j}\nabla_{i}\Theta
							=&\dfrac{\Theta}{nK}\dot{K}^{pq}\nabla_{q}\nabla_{p}h_{ij}-\dfrac{\Theta}{n}h^{pp}h^{qq}\nabla_{i}h_{pq}\nabla_{j}h_{pq}\\
							&+\dfrac{\Theta}{n}Hh_{ij}-\Theta h_{il}h_{jl}-\Theta\delta_{ij}+\dfrac{\Theta}{n}h^{pp}h_{ij}
							+\dfrac{\Theta}{u}\nabla_{j}\nabla_{i}u \\
							&+\dfrac{\Theta}{n^{2}}\nabla_{i}\log K\nabla_{j}\log K+\dfrac{\Theta}{nu}\nabla_{j}\log K\nabla_{i}u+\dfrac{\Theta}{nu}\nabla_{i}\log K\nabla_{j}u.\label{4.7}
					\end{aligned}
				\end{equation}
			Direct computation gives
				\begin{equation}
					\begin{aligned}
						\nabla_{j}\nabla_{i}Q
								&=\dfrac{\nabla_{j}\nabla_{i}v}{v}-\dfrac{\nabla_{j}v\nabla_{i}v}{v^{2}}-A\nabla_{j}\nabla_{i}u+B\nabla_{j}\nabla_{i}\rho \label{4.8} \\
								&=-h^{11}\nabla_{j}\nabla_{i}h_{11}+2h^{11}h^{pp}\nabla_{j}h_{1p}\nabla_{i}h_{1p}-(h^{11})^{2}\nabla_{j}h_{11}\nabla_{i}h_{11}\\
								&\ \ \ \ -A\nabla_{j}\nabla_{i}u+B\nabla_{j}\nabla_{i}\rho.
					\end{aligned}
				\end{equation}
			Recall (\ref{hij}). By (\ref{4.6}) and (\ref{4.7}), at $(x_{0},t_{0})$ we obtain
				\begin{equation}
					\begin{aligned}
						\partial_{t}Q&=\dfrac{\partial_{t}v}{v}-A\partial_{t}u+B\partial_{t}\rho \\
									 &=-h^{11}(\nabla_{1}\nabla_{1}\Theta+\Theta (h_{11})^{2}-\phi'(h_{11})^{2}+uh_{11}-\Theta)-A\partial_{t}u+B\partial_{t}\rho\\
									 &=-\dfrac{\Theta}{nK}\dot{K}^{pq}h^{11}\nabla_{q}\nabla_{p}h_{11}+\dfrac{\Theta}{n}h^{11}h^{pp}h^{qq}\nabla_{1}h_{pq}\nabla_{1}h_{pq}
									 -\dfrac{\Theta}{n}H+\Theta h_{11}  +\Theta h^{11}-\dfrac{\Theta}{n}h^{pp} \\
									 &\quad\quad-\dfrac{\Theta}{u}h^{11}\nabla_{1}\nabla_{1}u 
									 -\dfrac{\Theta}{n^{2}}(\nabla_{1}\log K)^{2}h^{11}-2\dfrac{\Theta}{nu}h^{11}\nabla_{1}\log K\nabla_{1}u-\Theta h_{11}+\phi'h_{11}\\
									 &\quad \quad-u+\Theta h^{11}-A\partial_{t}u+B\partial_{t}\rho \\
									 &\le -\dfrac{\Theta}{nK}\dot{K}^{pq}h^{11}\nabla_{q}\nabla_{p}h_{11}+\dfrac{\Theta}{n}h^{11}h^{pp}h^{qq}\nabla_{1}h_{pq}\nabla_{1}h_{pq}
									  +2\Theta h^{11} -\dfrac{\Theta}{u}h^{11}\nabla_{1}\nabla_{1}u\\
									 &\quad\quad  -\dfrac{\Theta}{n^{2}}(\nabla_{1}\log K)^{2}h^{11}-2\dfrac{\Theta}{nu}h^{11}\nabla_{1}\log K\nabla_{1}u+\phi'h_{11}-A\partial_{t}u+B\partial_{t}\rho. \label{4.9}
					\end{aligned}
				\end{equation}
			Substituting (\ref{4.8}) into (\ref{4.9}) and using Cauchy-Schwarz inequality, we get
				\begin{equation}
					\begin{aligned}
						\partial_{t}Q
								&\le\dfrac{\Theta}{nK}\dot{K}^{ij}\nabla_{j}\nabla_{i}Q-2\dfrac{\Theta}{n}h^{11}h^{ij}h^{pp}\nabla_{j}h_{1p}\nabla_{i}h_{1p}+\dfrac{\Theta}{n}(h^{11})^{2}h^{ij}\nabla_{j}h_{11}\nabla_{i}h_{11}\\
								&\quad\quad+A\dfrac{\Theta}{nK}\dot{K}^{ij}\nabla_{j}\nabla_{i}u
								-B\dfrac{\Theta}{n}h^{ij}\nabla_{j}\nabla_{i}\rho+\dfrac{\Theta}{n}h^{11}h^{pp}h^{qq}\nabla_{1}h_{pq}\nabla_{1}h_{pq}
								+2\Theta h^{11}\\
								&\quad\quad  -\dfrac{\Theta}{u}h^{11}\nabla_{1}\nabla_{1}u
								 +\dfrac{\Theta}{u^{2}}|\nabla_{1}u|^{2}h^{11}+\phi'h_{11}-A\partial_{t}u+B\partial_{t}\rho \\
								&\le -B\dfrac{\Theta}{n}h^{ij}\nabla_{j}\nabla_{i}\rho+2\Theta h^{11}+\dfrac{\Theta}{u^{2}}|\nabla_{1}u|^{2}h^{11}+\phi' h_{11}+B\partial_{t}\rho\\
								&\quad\quad+A\dfrac{\Theta}{nK}\dot{K}^{ij}\nabla_{j}\nabla_{i}u -\dfrac{\Theta}{u}h^{11}\nabla_{1}\nabla_{1}u-A\partial_{t}u.\label{4.10}
					\end{aligned}
				\end{equation}
			Recall (\ref{du}), we have $\nabla_{1}u=\phi\nabla_{1}\rho h_{11}$, hence $|\nabla_{1}u|^{2}\le\phi^{2}h_{11}^{2}$. Meanwhile, noting that $\frac{1}{K^{\frac{1}{n}}}\le h^{11}$, we divide (\ref{4.10}) by $\Theta=uK^{\frac{1}{n}}$ on both sides to obtain
				\begin{equation}
					\dfrac{\partial_{t}Q}{\Theta}\le -\dfrac{B}{n}h^{ij}\nabla_{j}\nabla_{i}\rho+2 h^{11}+\dfrac{\phi^{2}}{u^{2}}h_{11}+\dfrac{\phi'}{u}+B\dfrac{\partial_{t}\rho}{\Theta}+\dfrac{A}{nK}\dot{K}^{ij}\nabla_{j}\nabla_{i}u -\dfrac{\nabla_{1}\nabla_{1}u}{u}h^{11}-A\dfrac{\partial_{t}u}{\Theta}. \label{4.11}
				\end{equation}
			Recall (\ref{du}) again, we have $$\dfrac{\nabla_{1}\nabla_{1}u}{u}h^{11}=\dfrac{h^{11}}{u}\langle V,\nabla h_{11}\rangle+\dfrac{\phi'}{u}-h_{11}.$$
			Since $\nabla Q(x_{0},t_{0})=0$, we also have
				\begin{equation}
					\dfrac{\nabla v}{v}=A\nabla u-B\nabla \rho\quad \Rightarrow \quad h^{11}\nabla h_{11}=-A\nabla u+B\nabla\rho. \nonumber
				\end{equation}
			Hence we have
			\begin{equation}
			\dfrac{\nabla_{1}\nabla_{1}u}{u}h^{11}=-\dfrac{A}{u}\langle V,\nabla u \rangle+\dfrac{B}{u}\langle V,\nabla \rho\rangle+\dfrac{\phi'}{u}-h_{11}.\label{equ-h11}
		    \end{equation}
			Combining (\ref{4.3}) with (\ref{equ-h11}), we obtain
				\begin{equation}
					\begin{aligned}
						&\quad \quad \dfrac{A}{nK}\dot{K}^{ij}\nabla_{j}\nabla_{i}u -\dfrac{\nabla_{1}\nabla_{1}u}{u}h^{11}-A\dfrac{\partial_{t}u}{\Theta}\\
								&=-B\dfrac{\phi}{u}|\nabla\rho|^{2}-\dfrac{\phi'}{u}+h_{11}-A\dfrac{{\phi'}^{2}-\phi^{2}|\nabla\rho|^{2} }{\Theta}+2A\phi'-A\dfrac{u}{n}H\\
								&\le h_{11}-\dfrac{A}{\Theta}+2A\phi',\label{4.12}
					\end{aligned}
				\end{equation}
			where we used ${\phi'}^{2}-\phi^{2}|\nabla\rho|^{2}\ge{\phi'}^{2}-\phi^{2}=1$. 
			Differentiating $\langle V,V\rangle=\phi^{2}$, we obtain
			$$\phi\phi'\partial_{t}\rho=\langle\partial_{t}V,V\rangle=\phi'\langle\partial_{t}X,V\rangle.$$
			By (\ref{flow}), we have
				\begin{equation}
					\partial_{t}\rho=-\dfrac{\Theta}{w}+\dfrac{\phi'}{w}\le\dfrac{\phi'}{w}. \label{4.13}
				\end{equation}
			Substituting (\ref{4.12}) and (\ref{4.13}) into (\ref{4.11}), we get
				\begin{equation}
					\begin{aligned}
						\dfrac{\partial_{t}Q}{\Theta}&\le -\dfrac{B}{n}h^{ij}\nabla_{j}\nabla_{i}\rho+2 h^{11}+\dfrac{\phi^{2}}{u^{2}}h_{11}+B\dfrac{\phi'}{w\Theta}+h_{11}-\dfrac{A}{\Theta}+2A\phi'\\
						&\le -\dfrac{B}{n}h^{ij}\nabla_{j}\nabla_{i}\rho+C_{1}h_{11}+C_{2}h^{11}+AC_{3}
						\label{4.14}
					\end{aligned}
				\end{equation}
			for some $C_{1},C_{2},C_{3}>0$ depending only on $n$, $M_{0}$ and the uniform bounds of $\rho$ and provided $A=BC_{5}\ge2B\sup_{\mathbb{S}^n\times[0,T)}\dfrac{\phi'}{w}$, where $C_{5}$ depends only on the uniform upper and lower bounds of $\rho$.
			
			Finally, we wipe off some nonpositive terms. Recall (\ref{rho}), at $(x_{0},t_{0})$, we have
				\begin{equation}
					\label{equ-drho}
					|\nabla\rho|^{2}=\dfrac{\sum\limits_{i=0}^{n}\langle X_i,V\rangle^{2}}{\phi^{2}}=\dfrac{|V|^{2}}{\phi^{2}}-\dfrac{\langle\nu,V\rangle}{\phi^{2}}=1-\left(\dfrac{u}{\phi}\right)^{2}\le 1-C_{6},
				\end{equation}
			where $C_{6}$ depends on the lower bound of $u$ and the upper bound of $\rho$. Now substituting (\ref{rho}) into (\ref{4.14}) and using (\ref{equ-drho}), we obtain that at $(x_0,t_0)$,
				\begin{equation}
					\begin{aligned}
						0\le\dfrac{\partial_{t}Q}{\Theta}&\le-\dfrac{B}{n}h^{ii}\left(\dfrac{\phi'}{\phi}\left(1-\left(\nabla_{i}\rho\right)^{2}\right)-\dfrac{u}{\phi}h_{ii}\right)+C_{1}h_{11}+C_{2}h^{11}+AC_{3}\\
						&\le -BC_{0}h^{11}+BC_{4}+C_{1}\dfrac{1}{h^{11}}+C_{2}h^{11}+AC_{3} \label{4.16}
					\end{aligned}
				\end{equation}
			holds for some positive constants $C_{0},C_{1},C_{2},C_{3},C_{4}$ depending on the uniform upper and lower bounds of $\rho$ and $n$. We deduce from (\ref{4.16}) that if we choose $B=2\dfrac{C_{1}}{C_{0}}$ and $A=BC_{5}$, $\lambda(x_{0},t_{0})$ cannot be too large, i.e.,  $\lambda(x_{0},t_{0})$ has a uniform upper bound $C(C_{0},C_{1},C_{2},C_{3},C_{4},C_{5})$ independent of time $T_{0}$. 
			Thus, $Q(x_0,t_0)\leqslant C$, which implies that
			$$Q(x,t)\leqslant Q(x_0,t_0)\leqslant C.$$
			Therefore, $\lambda(x,t)\leqslant C$ for all $(x,t)\in M_t$, $t\in[0,T)$. It implies that $\kappa_i\geqslant c_0\ (i=1,2,\cdots,n)$. 
			
			Meanwhile, by Lemma \ref{lemma4.1}, we obtain that $C \ge K \geq c_{0}^{n-1}\kappa_{\max}$
			for some constant $C$ from Lemma \ref{lemma4.1}. Hence the principal curvatures are bounded from above. This completes
			the proof of Lemma \ref{lemma4.2}.
		\end{proof}
	
		By Lemma \ref{lemma4.2}, equation (\ref{equ-rho}) is uniformly parabolic. Combining the $C^0$ estimate and the gradient estimate in Lemmas \ref{lemma3.1} and \ref{lemma3.3}, one obtains the H\"older continuity of $\overline{\nabla}^2\rho$ 
		and $\partial_{t}\rho$ by \cite[Theorem 6]{fully2004}. Estimates for higher derivatives then follow by the
		standard regularity theory for uniformly parabolic equations. Hence we obtain the long
		time existence and regularity of solutions for the flow (\ref{1.1}).
		\begin{theorem}
			Let $M_{0}$ be a smooth, closed, uniformly convex hypersurface in $\mathbb{H}^{n+1}$
			which encloses the origin. Then the	flow \eqref{1.1} has a unique smooth, uniformly convex solution $M_{t}$ for all time $t\in [0,+\infty)$.
			
			Moreover, we have a priori estimates
				\begin{equation*}
					\Vert\rho\Vert_{C^{k,m}_{(\theta,t)}(\mathbb{S}^{n}\times[0,+\infty))}\le C_{k,m},
				\end{equation*}
			where $C_{k,m}>0$ depends only on $k$, $m$, and the geometry of $M_{0}$.
		\end{theorem}
	
\section{convergence}
\label{section5}
In this section, we will prove the smooth convergence and exponential convergence of the flow (\ref{1.1}).

\subsection{Smooth Convergence} We first consider subsequential convergence.
\begin{lemma}
	There exists a sequence $\{t_j\}$ such that when $t_j\to\infty$, $Q(t_j)$ monotonically decreasing converges to 0, where
	\begin{equation}
		Q(t) = \int_{M_t} \left( \phi'(\rho)-uK^{\frac{1}{n}} \right) \mathrm{d}\mu_t,
	\end{equation}
	and $M_{t_j}$ converges smoothly to a geodesic sphere centered at the origin.
\end{lemma}
\begin{proof}
	Note that $\dfrac{\mathrm{d}}{\mathrm{d}t}\mathscr{A}_{-1}(\Omega_t)=Q(t)$, then
	\begin{equation}
		\int_0^t Q(\tau) \mathrm{d}\tau = \mathscr{A}_{-1}(\Omega_t)-\mathscr{A}_{-1}(\Omega).
	\end{equation}
	By $K^{\frac{1}{n}}\leqslant \dfrac{H}{n}$ and Minkowski formulas \eqref{equ-minkf} when $k=0$, we know 
	$$Q(t)\geqslant \dfrac{1}{n}\int_{M_t} \left( n\phi'-uH \right) \mathrm{d}\mu_t = 0.$$ 
	By $C^0$ estimate, $\mathscr{A}_{-1}(\Omega_t)\leqslant |\mathbb{S}^n| \Phi(\max\limits_{\mathbb{S}^n} \rho_0)$. Thus $Q(t)$ is a nonnegative  $L^1(0,\infty)$ function. Then by Arzel\`a-Ascoli theorem and uniform regularity estimates, there exists $\{t_j\}$ such that $M_{t_j}$ converges smoothly to $M_{\infty}$ and the following holds:
	\begin{equation*}
		\int_{M_{\infty}} \left( \phi'(\rho) - u K^{\frac{1}{n}} \right) \mathrm{d}\mu_{\infty} =0.
	\end{equation*}
    Thus, we have
    \begin{equation}
    	\int_{M_{\infty}} u\left( K^{\frac{1}{n}} - \dfrac{H}{n} \right) \mathrm{d}\mu_{\infty}=0.
    \end{equation}
	Since $\frac{1}{C}\leqslant u \leqslant C$ by Corollary \ref{cor-starshapedness}, we have $K^{\frac{1}{n}}=\dfrac{H}{n}$ on $M_{\infty}$. Therefore, $M_{\infty}$ is a geodesic sphere.
	
	We now prove that $M_{\infty}$ is a geodesic sphere centered at the origin. Suppose the radius of $M_{\infty}$ is $\rho_{\infty}$, its principal curvatures are $\kappa_i=\dfrac{\phi'(\rho_{\infty})}{\phi(\rho_{\infty})}\ (i=1,\cdots,n)$. We argue by contradiction. Suppose $M_{\infty}$ is not centered at the origin, then at the point where $\rho$ attains its spatial maximum of $\theta\in\mathbb{S}^n$,
	\begin{equation}
		\dfrac{\partial\rho}{\partial t} = \phi'(\rho)-\phi(\rho)\dfrac{\phi'(\rho_{\infty})}{\phi(\rho_{\infty})} = \dfrac{\phi(\rho_{\infty}-\rho)}{\phi(\rho_{\infty})}<0,
	\end{equation}
	which contradicts that $M_{\infty}$ is the stationary solution to the flow (\ref{1.1}).
\end{proof}

Then we get smooth convergence.
\begin{proposition}
	There exists $\rho_{\infty}>0$ such that along the flow \eqref{1.1}, when $t\to\infty$, there holds $\|\rho(\cdot,t)-\rho_{\infty}\|_{C^k(\mathbb{S}^n)}\to 0$, $\forall\ k\in\mathbb{N}$.
\end{proposition}
\begin{proof}
	We have proved that there exist $t_j\to\infty$ and $\rho_{\infty}$ such that $\|\rho(\cdot,t_j)-\rho_{\infty}\|_{C^k(\mathbb{S}^n)}\to 0,\ \forall\ k\in\mathbb{N}$. Since $\rho_{\max}(t)=\max\limits_{\mathbb{S}^n}\rho(\theta,t)$ is non-increasing, $\rho_{\min}(t)=\min\limits_{\mathbb{S}^n}\rho(\theta,t)$ is non-decreasing, we have $\|\rho(\cdot,t)-\rho_{\infty}\|_{C^0(\mathbb{S}^n)}\to 0\ (t\to\infty)$.
	
	Recall the special case of Gagliardo-Nirenberg interpolation inequality (see \cite{Nirenberg1966AnEI}): For any positive integer $j$ and $k$, where $j<k$, and let $\delta=\frac{j}{k}$, there exists a constant $C(j,k)$, such that for any $f\in W^{k,2}(\mathbb{S}^n)$,
	\begin{equation}
		\left\|\dfrac{\partial^j}{\partial\theta^j} f\right\|_{L^2(\mathbb{S}^n)} \leqslant \left\|\dfrac{\partial^k}{\partial\theta^k} f\right\|_{L^2(\mathbb{S}^n)}^{\delta}\|f\|_{L^2(\mathbb{S}^n)}^{1-\delta}.
		\label{equ-gagli}
	\end{equation}
	Choose $f=\rho(\theta,t)-\rho_{\infty}$, we have
	\begin{equation}
		\left\|\dfrac{\partial^j}{\partial\theta^j} \rho(\theta,t)\right\|_{L^2(\mathbb{S}^n)} \leqslant \left\|\dfrac{\partial^k}{\partial\theta^k} \rho(\theta,t)\right\|_{L^2(\mathbb{S}^n)}^{\delta}\|\rho(\theta,t)-\rho_{\infty}\|_{L^2(\mathbb{S}^n)}^{1-\delta}.
	\end{equation}
	For any positive integer $j$, and any $\varepsilon>0$, for a large enough time $t$, by the boundedness of $\|\rho\|_{C^k(\mathbb{S}^n)}$ and $\|\rho(\cdot,t)-\rho_{\infty}\|_{C^0(\mathbb{S}^n)}<\varepsilon$, we know that $\left\| \dfrac{\partial^j}{\partial\theta^j} \rho(\theta,t) \right\|_{L^2(\mathbb{S}^n)}<C\varepsilon$. Then Sobolev embedding theorem implies that $\|\rho(\cdot,t)-\rho_{\infty}\|_{C^{j-1}(\mathbb{S}^n)}<C\varepsilon$. Hence we prove that $\rho(\cdot,t)$ converges smoothly to $\rho_{\infty}$ as $t\to\infty$.
\end{proof}
\subsection{Exponential convergence} We now prove the exponential convergence of the flow (\ref{1.1}).

Let's first derive the evolution equation of $\dfrac{|\overline{\nabla}\gamma|^2}{2}$ along (\ref{equ-gammat}).

\begin{lemma}
	Let $\xi=\dfrac{|\overline{\nabla}\gamma|^2}{2}$, then along the flow \eqref{equ-gammat}, there holds
	\begin{eqnarray}
		\mathcal{L}\xi &=& - \dfrac{1}{n}K^{\frac{1}{n}-1}\dot{K}^{ij}\left[ \dfrac{1}{\phi w^3}\xi^i \xi_j- \dfrac{3}{\phi w^5} \gamma^k \xi_k \gamma^i\xi_j - \dfrac{2\phi'}{\phi w^3} \xi \gamma^i\xi_j -\dfrac{\phi'}{\phi w^3}\gamma^k\xi_k \delta^i_j + \dfrac{1}{\phi w^3} \gamma^k\xi_k\gamma^i_j \right] \notag\\
		&& + \dfrac{\phi'}{\phi\omega} \gamma^k\xi_k -\dfrac{1}{n}K^{\frac{1}{n}-1}\dot{K}^{ij}\left[ \dfrac{1}{\phi w} \gamma^i_k\gamma^k_j - \dfrac{1}{\phi w}\gamma^i\gamma_j + \dfrac{2\phi'}{\phi w}\xi\gamma^i_j \right] - \dfrac{2w\xi}{\phi},
	\end{eqnarray}
    where the operator $\mathcal{L}$ is defined by
    \begin{equation}
    	\mathcal{L}\xi = \partial_t\xi - \dfrac{1}{n\phi w^3}K^{\frac{1}{n}-1}\dot{K}^{ij} (w^2 e^{ik}-\gamma^i\gamma^k)\overline{\nabla}_k\overline{\nabla}_j\xi,
    \end{equation}
    and we denote $\dot{K}^{ij}=\dfrac{\partial K}{\partial h^i_j}$, $\gamma^i_j=e^{ik}\gamma_{kj}$. Here $e_{ij}$ is the standard metric on $\mathbb{S}^n$.
\end{lemma}
\begin{proof}
	Recall that
	\begin{equation}
		w=\sqrt{1+|\overline{\nabla}\gamma|^2},\ \ \overline{\nabla} w = \dfrac{\overline{\nabla}\xi}{w}.
	\end{equation}
	By \eqref{equ-gammat}, we have
	\begin{eqnarray}
		\partial_t\left( \dfrac{|\overline{\nabla}\gamma|^2}{2} \right) &=& \overline{\nabla}\gamma\cdot \overline{\nabla}\gamma_t \notag\\
		&=& \overline{\nabla}\gamma \cdot \overline{\nabla} \left( \dfrac{\phi'}{\phi}w-K^{\frac{1}{n}} \right) \notag\\
		&=& \overline{\nabla}\gamma \cdot \left[ \dfrac{w\overline{\nabla} \phi'+\phi'\overline{\nabla}w}{\phi}-\dfrac{\phi'w\overline{\nabla}\phi}{\phi^2} - \overline{\nabla}K^{\frac{1}{n}} \right].
	\end{eqnarray}
    Note that,
    \begin{eqnarray}
    	&&\overline{\nabla}\phi' = \phi''\phi\overline{\nabla}\gamma,\ \ \overline{\nabla}\phi=\phi'\phi\overline{\nabla}\gamma,\notag\\
    	&&\overline{\nabla}K^{\frac{1}{n}} = \dfrac{1}{n}K^{\frac{1}{n}-1}\dot{K}^{ij}\overline{\nabla}h^i_j,
    \end{eqnarray}
    and combining $(\phi')^2-\phi''\phi=1$, we have
    \begin{eqnarray}
    	\partial_t \xi = -\dfrac{2w\xi}{\phi} + \dfrac{\phi'}{\phi w}\overline{\nabla}\gamma \cdot \overline{\nabla}\xi - \dfrac{1}{n}K^{\frac{1}{n}-1}\dot{K}^{ij}\overline{\nabla}\gamma\cdot \overline{\nabla}h^i_j.  \label{equ-xit}
    \end{eqnarray}
    Recall that by (\ref{hi^j}) and (\ref{equ-gammap}),
    \begin{equation}
    	h^i_j = \dfrac{1}{\phi w}\left( \phi' \delta^i_j-\gamma^i_j + \dfrac{1}{w^2}\gamma^i\xi_j \right).
    \end{equation}
    By a direct computation,
    \begin{eqnarray}
    	\overline{\nabla}_k h^i_j &=& \dfrac{1}{\phi w}\left( \phi''\phi\gamma_k\delta^i_j-\overline{\nabla}_k\gamma^i_j+\dfrac{1}{w^2}\gamma^i_k\xi_j + \dfrac{1}{w^2} \gamma^i\xi_{jk}-\dfrac{2}{w^4}\xi_k\gamma^i\xi_j \right)\notag\\
    	&&-\dfrac{w^2\phi'\phi\gamma_k+\phi\xi_k}{\phi^2 w^3}\left( \phi'\delta^i_j - \gamma^i_j + \dfrac{1}{w^2} \gamma^i\xi_j \right).
    \end{eqnarray}
    Using Ricci identity to commute the covariant derivatives on the sphere $\mathbb{S}^n$, we obtain
    \begin{eqnarray}
    	\gamma^k\overline{\nabla}_k\gamma^i_j &=& \gamma^k\left( \overline{\nabla}_j\gamma^i_k + \gamma_j\delta^i_k- \gamma_k\delta^i_j \right)\notag\\
    	&=& \xi^i_j - \gamma^i_k\gamma^k_j + \gamma^i\gamma_j - 2\xi\delta^i_j.
    \end{eqnarray}
    Hence
    \begin{eqnarray}
    	\gamma^k\overline{\nabla}_k h^i_j &=& \dfrac{1}{\phi w^3}\left( \gamma^i\gamma^k - w^2 e^{ik}  \right)\xi_{kj}  + \dfrac{1}{\phi w^3}\xi^i \xi_j- \dfrac{3}{\phi w^5} \gamma^k \xi_k \gamma^i\xi_j - \dfrac{2\phi'}{\phi w^3} \xi \gamma^i\xi_j \notag\\
    	&&-\dfrac{\phi'}{\phi w^3}\gamma^k\xi_k \delta^i_j + \dfrac{1}{\phi w^3} \gamma^k\xi_k\gamma^i_j + \dfrac{1}{\phi w} \gamma^i_k\gamma^k_j - \dfrac{1}{\phi w}\gamma^i\gamma_j + \dfrac{2\phi'}{\phi w}\xi\gamma^i_j. \label{equ-gak}
    \end{eqnarray}
    Substituting (\ref{equ-gak}) into (\ref{equ-xit}), we obtain the result.
\end{proof}

\begin{proposition}
	Let $\gamma(\theta,t),\ (\theta,t)\in\mathbb{S}^n\times[0,+\infty)$ be a smooth solution of initial value problem \eqref{equ-gammat}, and $\rho_{\infty}$ be the radius of limit geodesic sphere. Then for any positive constant $\alpha$ satisfying
	\begin{equation*}
		\alpha<\dfrac{n-1}{n\phi(\rho_{\infty})}
	\end{equation*}
    which depends only on the initial hypersurface $M_0$, we have
    \begin{equation*}
    	|\overline{\nabla}\gamma|^2(\theta,t)\leqslant C\mathrm{e}^{-\alpha t},\ \ \forall\ (\theta,t)\in\mathbb{S}^n\times[0,+\infty),
    \end{equation*}
    where $C$ is a positive constant depending only on $M_0$ and $\alpha$.
\end{proposition}
\begin{proof}
	Since $M_t$ converges smoothly to a geodesic sphere as $t\to\infty$, there exists $T_0>0$ such that for any $t>T_0$,
	\begin{equation}
		\max\limits_{1\leqslant i<j\leqslant n}|\kappa_i-\kappa_j|\leqslant\varepsilon \label{equ-ka}
	\end{equation}
	holds on $M_t$ by the regularity estimates of $M_t$. We only need to obtain the estimate of $\xi=\frac{|\overline{\nabla}\gamma|^2}{2}$ for large enough time $t>T_0$. We also assume that 
	$$\xi\leqslant \varepsilon <1 \  \ \ \mbox{and}\ \ \   \dfrac{1}{n}K^{\frac{1}{n}-1}\dot{K}^{ii}\in (\frac{1}{n}-\varepsilon,\frac{1}{n}+\varepsilon)$$
	 hold on $M_t$ for each time $t>T_0$.
	
	For any time $t_0>T_0$, let $\theta_{t_0}\in\mathbb{S}^n$ be the spatial maximum point of $\xi$ at $t=t_0$. Then at $(\theta_{t_0},t_0)$, we have
	\begin{equation}
		\overline{\nabla}\xi = 0,\ \ \left(\overline{\nabla}^2\xi\right)\leqslant 0.
	\end{equation}
    Note that $\dot{K}^{ij}$ and $(w^2e^{ik}-\gamma^i\gamma^k)$ are positive definite, we get that at $(\theta_{t_0},t_0)$,
    \begin{equation}
    	\dfrac{\partial}{\partial t}\xi \leqslant \dfrac{1}{n\phi w}K^{\frac{1}{n}-1}\dot{K}^{ij}\gamma^i\gamma_j-\dfrac{2w\xi}{\phi}-\dfrac{2\phi'}{\phi w}\dfrac{1}{n}K^{\frac{1}{n}-1}\dot{K}^{ij}\gamma^i_j \xi.  \label{equ-ine}
    \end{equation}
    At this point, we choose orthonormal frame such that
    \begin{equation*}
    	\gamma_1 = |\overline{\nabla}\gamma|, \ \ \gamma_j=0,\ \ 2\leqslant j\leqslant n,
    \end{equation*}
    and $(\gamma_{ij})$ is diagonal with $\gamma_{1j}=0$ for all $j=1,2,\cdots, n$. Then the Weingarten matrix
    \begin{equation}
    	h^i_j = \dfrac{1}{\phi w}\left( -\gamma_{ij} + \phi'\delta^i_j \right)
    \end{equation}
    is diagonalized at $(\theta_{t_0},t_0)$, and $\gamma_{11}=0$, $|\gamma_{ii}|=|\gamma_{ii}-\gamma_{11}|\leqslant \varepsilon \phi w$ for $2\leqslant i\leqslant n$ by (\ref{equ-ka}). Now inequality (\ref{equ-ine}) becomes
    \begin{eqnarray}
    	\dfrac{\partial}{\partial t}\xi &\leqslant& \dfrac{2}{n\phi w}K^{\frac{1}{n}-1}\dot{K}^{11}\xi - \dfrac{2w\xi}{\phi} - \dfrac{2\phi'}{\phi w}\left(\dfrac{1}{n}K^{\frac{1}{n}-1}\sum\limits_{i=2}^n\dot{K}^{ii}\gamma_{ii}\right)\xi\notag\\
    	&\leqslant& -\left( \dfrac{n-1}{n\phi}-\tilde{\varepsilon} \right)\xi,
    \end{eqnarray}
    where $\tilde{\varepsilon}\leqslant C\varepsilon$ is a small constant. Let
    \begin{equation*}
        \alpha_0 = \dfrac{n-1}{n\phi(\rho_{\infty})}.
    \end{equation*}
    Then for any $\alpha-\alpha_0$, there exists a time $T_0$ depending on $M_0$ and $\alpha_0-\alpha$ such that for all $t>T_0$, there holds
    \begin{equation*}
    	\dfrac{\partial}{\partial t}\xi\leqslant -\alpha\xi,\ \ t\in[T_0,+\infty)
    \end{equation*}
    at the spatial maximum of $\xi$. Integrating the above inequality, we obtain
    \begin{equation*}
    	\xi(t)\leqslant \xi(T_0)\mathrm{e}^{-\alpha(t-T_0)}
    \end{equation*}
    for all time $t\geqslant T_0$. This completes the proof.
\end{proof}

By Gagliardo-Nirenberg interpolation inequality again, we get the exponential convergence of the flow (\ref{1.1}) in $C^k$ norm ($k\in\mathbb{N}$). Hence, we complete the proof of Theorem \ref{thm1.1}.

\section{Alexandrov-Fenchel inequalities}
\label{section6}
As an application, we will compare quermassintegrals $\mathscr{A}_{n-2}(\Omega)$ and $\mathscr{A}_{-1}(\Omega)$ for a smooth convex domain $\Omega\subset\mathbb{H}^{n+1}\ (n\geqslant 2)$. Now let us prove Theorem \ref{thm-afineq}.

\begin{proof}[Proof of Theorem \ref{thm-afineq}]
	We divide the proof into two steps.
	
	{\bf Step 1.} Assume that $M$ is a smooth uniformly convex hypersurface enclosing a domain $\Omega$. We evolve $M$ along the flow (\ref{1.1}) and obtain a family of uniformly convex hypersurface $M_t$ for all time by Theorem \ref{thm1.1}. As $t\to\infty$, $M_t$ converges smoothly and exponentially to a geodesic sphere $M_{\infty}$ centered at the origin with radius $r_{\infty}$. Let $\Omega_{\infty}$ be the domain enclosed by $M_{\infty}$. On the one hand, along the flow (\ref{1.1}),
	\begin{eqnarray}
		\dfrac{\mathrm{d}}{\mathrm{d}t}\mathscr{A}_{n-2}(\Omega_t) &=& (n-1) \int_{M_t} \left( \phi'(\rho)-uK^{\frac{1}{n}} \right)\sigma_{n-1}(\kappa) \mathrm{d}\mu_t \notag\\
		&\leqslant & (n-1) \int_{M_t} \left( \phi'(\rho)\sigma_{n-1}(\kappa) - nu\sigma_n(\kappa) \right) \mathrm{d}\mu_t = 0, \label{equ-anm2}
	\end{eqnarray}
	where we used $\sigma_{n-1}(\kappa) K^{\frac{1}{n}}\geqslant n\sigma_n(\kappa)$ and Minkowski formulas (\ref{equ-minkf}) with $k=n-1$. Equality holds in \eqref{equ-anm2} if and only if $\Omega_t$ is a geodesic ball. On the other hand,
	\begin{eqnarray}
		\dfrac{\mathrm{d}}{\mathrm{d} t} \mathscr{A}_{-1} (\Omega_t) &=& \int_{M_t} \left( \phi'(\rho)-uK^{\frac{1}{n}} \right) \mathrm{d}\mu_t \notag\\
		&\geqslant&  \dfrac{1}{n} \int_{M_t} (n\phi'(\rho)-uH) \mathrm{d}\mu_t = 0, \label{equ-am1}
	\end{eqnarray}
	where we used $K^{\frac{1}{n}}\leqslant\dfrac{H}{n}$ and Minkowski formulas (\ref{equ-minkf}) with $k=0$. Equality holds in \eqref{equ-am1} if and only if $\Omega_t$ is a geodesic ball. Thus,
	\begin{equation}
		\mathscr{A}_{n-2}(\Omega) \geqslant \mathscr{A}_{n-2}(\Omega_{\infty}) = \xi_{n-2}(r_{\infty}),
	\end{equation}
	and
	\begin{equation}
		\mathscr{A}_{-1}(\Omega) \leqslant \mathscr{A}_{-1}(\Omega_{\infty}) = \xi_{-1}(r_{\infty}).
	\end{equation}
	Since $\xi_{-1}$ is strictly increasing, its inverse function $\xi_{-1}^{-1}$ is well-defined. Let $\xi_{n-2,-1}=\xi_{n-2}\circ \xi_{-1}^{-1}$, then inequality (\ref{equ-afine}) for uniformly convex $M$ holds. Equality in (\ref{equ-afine}) holds if and only if
	$$\dfrac{\mathrm{d}}{\mathrm{d} t}\mathscr{A}_{n-2}(\Omega_t)=\dfrac{\mathrm{d}}{\mathrm{d}t}\mathscr{A}_{-1}(\Omega_t)=0$$
	holds for all time, which is equivalent to that $\Omega$ is a geodesic ball.
	
	{\bf Step 2.} Let $M$ be a smooth convex hypersurface enclosing the domain $\Omega$, then there exists a sequence of uniformly convex hypersurfaces $M_{\varepsilon}$ such that $M_{\varepsilon}\to M$ as $\varepsilon\to 0$. In fact, we can project $\Omega$ into $B_1(O)\subset\mathbb{R}^{n+1}$ by Klein model (see \cite[Section 5]{Andrews2018QuermassintegralPC} and \cite[Section 3]{WYZ2023}).
	
	For convenience, we review the Klein model here. Denote $\mathbb{R}^{1,n+1}$ be the Minkowski space time, which is the vector space $\mathbb{R}^{n+2}$ equipped with Minkowski spacetime metric $\langle\cdot,\cdot \rangle$: For any vector $X=(X_0,X_1,X_2,\cdots,X_{n+1})\in\mathbb{R}^{n+2}$,
	$$\langle X,X \rangle=-X_0^2+\sum\limits_{i=1}^{n+1}X_i^2.$$ 
	The hyperbolic space $\mathbb{H}^{n+1}$ is then viewed as a hypersurface in $\mathbb{R}^{1,n+1}$,
	\begin{equation}
		\mathbb{H}^{n+1} = \left\{ X\in\mathbb{R}^{1,n+1}:\langle X,X \rangle=-1,\ X_0>0 \right\}.
	\end{equation}
	The Klein model parametrizes the hyperbolic space by using the unit disc, which induces a projection from an embedding $X:M^n \to\mathbb{H}^{n+1}$ to an embedding $Y:M^n\to B_1(O)\subset \mathbb{R}^{n+1}$ by
	\begin{equation}
		X=\dfrac{(1,Y)}{\sqrt{1-|Y|^2}}.
	\end{equation}
	Let $\nu\in T\mathbb{H}^{n+1}$, $h_{ij}^X$ and $N\in\mathbb{R}^{n+1}$, $h_{ij}^{Y}$ denote the unit normal vector, the second fundamental form of $X(M)\subset\mathbb{H}^{n+1}$ and $Y(M)\subset\mathbb{R}^{n+1}$ respectively. We have the relation:
	\begin{equation}
		h_{ij}^{X} = \dfrac{h_{ij}^{Y}}{\sqrt{(1-|Y|^2)(1-\langle N,Y \rangle^2)}}. \label{equ-second}
	\end{equation}
	
	By (\ref{equ-second}) we know the image $\hat{\Omega}$ under projection is also convex. By mean curvature flow, there exists a sequence of uniformly convex domains $\hat{\Omega}_{\varepsilon}$ that approximate $\hat{\Omega}$. Since the projection is a diffeomorphism, we find a family of uniformly convex domains $\Omega_{\varepsilon}\subset\mathbb{H}^{n+1}$ that approximate $\Omega$. Inequality (\ref{equ-afine}) for convex $M$ follows by letting $\varepsilon\to 0$ in the inequality (\ref{equ-afine}) for the uniformly convex $M_{\varepsilon}$.
	
	Now assume $M$ is convex and equality in (\ref{equ-afine}) holds. Let
	\begin{equation}
		M_{+} = \{p\in M: \kappa_i(p)>0,\ 1\leqslant i\leqslant n\}.
	\end{equation}
	Since a compact embedded hypersurface in $\mathbb{H}^{n+1}$ contains at least one elliptic point, $M_+$ is a nonempty open set. We claim that $M_+$ is also closed. Following \cite{Guan2009TheQI}, we consider a normal variation of $M$
	\begin{equation}
		\partial_t X = -\varphi\nu,
	\end{equation} 
	where $\varphi\in C_c^{\infty}(M_+)$ is a smooth function with compact support in $M_+$. For sufficiently small time $t\in(-\varepsilon,\varepsilon)$, $M_t$ is convex. Set
	\begin{equation}
		Q(t) = \mathscr{A}_{n-2}(\Omega_t) - \xi_{n-2,-1}\left( \mathscr{A}_{-1}(\Omega_t) \right),
	\end{equation}
	then
	\begin{equation}
		Q(0)=0,\ \ Q(t)\geqslant 0,\ \forall\ t\in(-\varepsilon,\varepsilon),
	\end{equation}
	which implies that $Q(t)$ attains its minimum at $t=0$. Thus,
	\begin{eqnarray}
		0 &=& \left.\dfrac{\mathrm{d}}{\mathrm{d}t}\right|_{t=0}\left[ \mathscr{A}_{n-2}(\Omega_t) - \xi_{n-2,-1}\left( \mathscr{A}_{-1}(\Omega_t) \right) \right] \notag\\
		&=& (n-1)\int_M \sigma_{n-1}(\kappa) (-\varphi) \mathrm{d}\mu -\xi_{n-2,-1}'\left(  \mathscr{A}_{-1}(\Omega)\right) \int_{M} (-\varphi) \mathrm{d}\mu \notag\\
		&=& - \int_M \left[ (n-1)\sigma_{n-1}(\kappa) -\xi_{n-2,-1}'\left( \mathscr{A}_{-1}(\Omega) \right) \right] \varphi\mathrm{d}\mu,\ \ \forall\ \varphi\in C_c^{\infty}(M_+).
	\end{eqnarray}
	In particular,
	\begin{equation}
		M_+ = \left\{p\in M:(n-1)\sigma_{n-1}(\kappa)=\xi_{n-2,-1}'\left( \mathscr{A}_{-1}(\Omega) \right)\right\}.
	\end{equation}
	Therefore, $M_+$ is closed in $M$ and the connectedness of $M$ implies $M_+=M$. We conclude that $M$ is uniformly convex. By the equality case in {\bf Step 1}, we obtain that $\Omega$ is a geodesic ball. Hence we complete the proof of Theorem \ref{thm-afineq}.
\end{proof}

\bibliographystyle{plainnat}
\bibliography{reference.bib}

\end{document}